\documentclass[a4paper,11pt]{amsart}

\usepackage{amsmath, amssymb, amsthm}
\usepackage{verbatim}
\usepackage{graphicx}

\usepackage[top=3cm, bottom=3cm, left=2.5cm, right=2.5cm]{geometry}

\usepackage[colorlinks=true, linkcolor=blue, citecolor = blue]{hyperref}
\usepackage{mathrsfs}

\newtheoremstyle{example}{\topsep}{\topsep}%
     {}
     {}
     {\bfseries}
     {.}
     {  }
     {}

\newtheorem{thm}{Theorem}[section]
\newtheorem*{thm*}{Theorem}

\newtheorem{defn}[thm]{Definition}
\newtheorem*{def*}{Definition}
\newtheorem{assumpt}{Assumption}[section]
\theoremstyle{example}

\newtheorem{rmk}{Remark}[section]

\newcommand{\eps}{\varepsilon}

\newcommand{\om}{\omega}
\newcommand{\Om}{\Omega}
\newcommand{\fil}{\mathcal{F}}

\newcommand{\expt}{\mathbb{E}}
\newcommand{\mbbP}{\mathbb{P}}

\newcommand{\R}{\mathbb{R}}

\newcommand{\C}{\mathcal{C}}
\newcommand{\Ch}{\hat{\C}}
\newcommand{\gen}{\mathcal{A}}
\newcommand{\dom}{\mathcal{D}}
\newcommand{\Ah}{\hat{\gen}}
\newcommand{\Th}{\hat{T}}
\newcommand{\Ind}{\mathbf{1}_{\{0\}}}
\newcommand{\la}{\langle}
\newcommand{\ra}{\rangle}
\newcommand{\pih}{\hat{\pi}}
\newcommand{\Qhl}{\hat{Q}_{\Lambda}}
\newcommand{\Psiz}{\hat{\Psi}}

\newcommand{\pj}{\Pi}
\newcommand{\icond}{\zeta}
\newcommand{\stopt}{\mathfrak{e}}

\newcommand{\zed}{\mathfrak{z}^\eps}

\newcommand{\gnoise}{\xi}

\newcommand{\nstsp}{\mathbf{M}}
\newcommand{\noisegent}{\mathfrak{G}}
\newcommand{\igen}{\mathcal{L}}
\newcommand{\cmp}{\mathfrak{X}} 
\newcommand{\clp}{\cmp^{\dag}}

\newcommand{\Qh}{\hat{Q}}

\newcommand{\Sp}{\mathbb{S}}
\newcommand{\teps}{\tau^\eps}
\newcommand{\nablaz}{\nabla^{\mathfrak{z}}}
\newcommand{\nablay}{\nabla^{\mathfrak{y}}}
\newcommand{\trunc}{\vartheta}
\newcommand{\zedx}{\mathfrak{z}}

\newcommand{\xiavg}{\Xi}
\newcommand{\yavg}{\mathbb{Y}}
\newcommand{\tauavg}{\mathbb{T}}

\newcommand{\zedaux}[1]{\mathfrak{z}^{\eps,#1}}

\newcommand{\sppm}{\mathcal{P}}

\title[Noise perturbations of DDE at Hopf bifurcation]{Exponentially-ergodic markovian noise perturbations of delay differential equations at Hopf bifurcation.}
\author{N. Lingala}
\author{N. Sri Namachchivaya}
\author{V. Wihstutz}



\begin{document}


\begin{abstract} 
We consider noise perturbations of delay differential equations (DDE) experiencing Hopf bifurcation. The noise is assumed to be exponentially ergodic, i.e. transition density converges to stationary density exponentially fast uniformly in the initial condition. We show that, under an appropriate change of time scale, as the strength of the perturbations decreases to zero, the law of the critical eigenmodes converges to the law of a diffusion process (without delay). We prove the result only for scalar DDE. For vector-valued DDE without proofs see \cite{PRE}.
\end{abstract}

\maketitle


\section{Introduction}\label{sec:intro}
Delay differential equations (DDEs) arise in a variety of areas such as manufacturing systems, biological systems, and control systems.  Deterministic DDEs have been the focus of intense study by many authors in the past three decades---see \cite{Halebook, Diekmanbook} and the references therein. In some of these systems, variation of a parameter would result in loss of stability through Hopf bifurcation. 
For example, oscillators of the form 
\begin{align}\label{eq:osci_chat_ex}
\ddot{q}(t)+2\zeta \dot{q}(t)+q(t)=-\kappa_1 [q(t)-q(t-\tau)]+\kappa_2[q(t)-q(t-\tau)]^2
\end{align}
arise in machining processes---here $q$ represents the position of a tool cutting a workpiece that is rotating with time period $\tau$, and $\kappa_i$ depend on the width of the cut. There exists a threshold $\kappa$ beyond which the fixed point $q=0$ loses stability through Hopf bifurcation and oscillations arise \cite{Gabor1KalNag}. This oscillatory behaviour is called regenerative chatter and results in poor surface finish of the workpiece. Hopf bifurcation is also found in biological systems---for example \cite{Longtin} discusses a model for oscillations in the area of eye-pupil as a response to incident light. The model has similar qualitative behaviour to Mackey-Glass equation \cite{Mac_Gla_scholarpedia}
\begin{equation*}
\frac{d}{dt}\hat{x}(t)=-\alpha \hat{x}(t)+\frac{c\theta^n}{\theta^n+\hat{x}^n(t-\tau)},
\end{equation*}
which exhibits a Hopf bifurcation as the parameter $n$ is varied.

Typically these systems are also influenced by noise, for example, inhomogenity in the material properties of workpiece in machining processes \cite{Buck_Kusk}, and unmodeled dynamics in biological systems. Therefore, it is important to study the effect of noise in the models of such systems. The aim of this paper is to study the effect of nonlinear and random perturbations on those delay systems whose fixed points are on the verge of a Hopf bifurcation. With appropriate scaling of coordinates, the dynamics close to the fixed point can be casted in the form of a linear DDE perturbed by small noise and small nonlinearities.
First we briefly describe the mathematical set-up. Statements would be proved only for scalar DDE. 

Let $x(t)$ be a $\R$-valued process governed by a DDE with maximum delay $r$. The evolution of $x$ at each time $t$ requires the history of the process in the time interval $[t-r,t]$. So, the state space can be taken as $\C:=C([-r,0];\R)$, the space of continuous functions on $[-r,0]$. Equipped with sup norm, $||\eta||=\sup_{\theta \in [-r,0]}|\eta(\theta)|$, the space $\C$ is a Banach space. At each time $t$, denote the $[t-r,t]$ segment of $x$ as $\pj_t x$, i.e. $\pj_t x \in \C$ and
$$\pj_t x(\theta)=x(t+\theta), \quad \text{ for }\theta \in [-r,0].$$
Now, a linear DDE can be represented in the following form
\begin{align}\label{eq:detDDE}
\begin{cases}\dot{x}(t)=L_0(\pj_t x), \qquad t\geq 0, \\
\pj_0x =\icond \in \C,
\end{cases}
\end{align}
where $L_0:\C\to \R$ is a continuous linear mapping on $\C$ and $\icond$ is the initial history required. For every such $L_0$ there exists a bounded function $\mu:[-r,0]\to \R$, continuous from the left on $(-r,0)$ and normalized with $\mu(0)=0$, such that 
\begin{align}\label{eq:L0matrixrep}
L_0\eta = \int_{[-r,0]}d\mu(\theta)\eta(\theta), \quad \forall \eta \in \C.
\end{align}
To reflect the Hopf bifurcation scenario, we consider operators $L_0$ which are such that the unperturbed system \eqref{eq:detDDE} is on the verge of instability, i.e., we assume $L_0$ satisfies the following assumption.
\begin{assumpt}\label{ass:assumptondetsys}
Define $\Delta(\lambda):=\lambda I_{n\times n}-L_0(e^{\lambda \cdot})\,=\,\lambda I_{n\times n}-\int_{[-r,0]}d\mu(\theta)e^{\lambda \theta}.$ The characteristic equation 
\begin{align}\label{eq:chareq}
det(\Delta(\lambda))=0, \qquad \lambda \in \mathbb{C}
\end{align}
has a pair of purely imaginary solutions $\pm i\om_c$ (critical eigenvalues) and all other solutions have negative real parts (stable eigenvalues).
\end{assumpt}
The object of study in this article are equations of the form
\begin{align}\label{eq:detDDE_pert_gennoise_intro}
\begin{cases}dx(t)=L_0(\pj_t x)dt+\eps G_q(\pj_t x)dt+\eps^2 G(\pj_t x)dt+\eps \sigma(\xi(t)) F(\pj_t x)dt, \quad t\geq 0, \\
\pj_0x =\icond \in \C,
\end{cases}
\end{align}
where $F,G,G_q:\C \to \R$ satisfy assumption \ref{assmp:assmp_on_FGGq}, $\xi$ is a noise process satisfying assumption \ref{assmp:assmp_on_noise} and $\sigma:\nstsp\to \R$ is a bounded mean-zero function of the noise $\xi$. For example, one can have $\xi$ as a finite-state markov chain. The coefficient $G_q$ is assumed to satisfy a centering condition that would be specified later in assumption \ref{ass:assumptonGqcentering}.
\begin{assumpt}\label{assmp:assmp_on_FGGq}
The functions $F,G,G_q$ have atmost linear growth.
$$|F(\eta)|\leq C(1+||\eta||), \quad |G(\eta)|\leq C(1+||\eta||), \quad |G_q(\eta)|\leq C(1+||\eta||), \quad \forall \eta \in \C.$$
The functions $F,G,G_q$ possess three bounded derivatives. 
\end{assumpt}
\begin{assumpt}\label{assmp:assmp_on_noise}
The noise $\xi$ is a $\nstsp$-valued Markov process with the transition function $\nu$ given by
 $$\nu(t,\xi,B)=\mbbP\{\xi_t\in B \,| \,\xi_0=\xi\}$$
for $B$ a borel subset of $\nstsp$.
The noise is exponentially erogdic, i.e., there exist a unique invariant probability measure $\bar{\nu}$ and positive constants $c_1$ and $c_2$ such that for all $t\geq 0$,
$$\sup_{\xi\in \nstsp}\int_{\nstsp}|\nu(t,\xi,d\zeta)-\bar{\nu}(d\zeta)| \leq c_1e^{-c_2t}.$$
The transition semigroup is Feller with infinitesimal generator denoted by $\noisegent$.  The function $\sigma$ is bounded, $\sigma(\cdot) \in dom(\noisegent)$ and such that $\int_{\nstsp}\sigma(\zeta)\bar{\nu}(d\zeta)=0$.
\end{assumpt}
When studying the effect of small noise perturbations on DDE whose fixed point is on the verge of Hopf bifurcation, the dynamics close to the fixed point can be casted in the above forms after appropriate scaling of coordinates. For example, consider $\dot{\tilde{x}}=\kappa \tilde{x}(t-1)-\tilde{x}^3(t)$. When $\kappa=-\frac{\pi}{2}$ the fixed point $\tilde{x}=0$ is on the verge of instability. Suppose $\kappa$ has small perturbations about $-\frac{\pi}{2}$ according to $\kappa(t)=-\frac{\pi}{2} +\eps \sigma(\gnoise(t)) +\eps^2$ where $\gnoise$ is a noise. Then, zooming close to the zero fixed point, $x(t):=\eps^{-1} \tilde{x}(t)$ can be put in the form  \eqref{eq:detDDE_pert_gennoise_intro} with $L_0(\eta)=-\frac{\pi}{2}\eta(-1)$, $F(\eta)=\eta(-1)$ and $G(\eta)=-\eta^3(0)+\eta(-1)$.

When $\eps=0$ in \eqref{eq:detDDE_pert_gennoise_intro}, using spectral theory \cite{Halebook}, $\C$ can be decomposed as $\C=P\oplus Q$ where $P$ is a two-dimensional space determined by the (eigenspaces associated with the) pair of critical eigenvalues. The projections of $\pj_tx$ onto $P$ and $Q$ are uncoupled. In $Q$ the dynamics is governed by the stable eigenvalues, and hence the sup-norm of the $Q$-projection of $\pj_tx$ decays to zero exponentially fast as $t\to \infty$. The space $P$ is two-dimensional and a basis $\Phi$ can be chosen for it. Let $(z_1(t), z_2(t))$ be the coordinates of $P$-projection of $\pj_tx$ with respect to the basis $\Phi$. Then the dynamics of $z$ is a pure rotation with frequency $\om_c$ and constant amplitude. 

 When the perturbation is added, i.e. $\eps>0$, the dynamics of $\pj_tx$ on $P$ and $Q$ is coupled. The amplitude of the $Q$-projection decays exponentially fast to a ``strip'' of $O(\eps)$  and the dynamics of $z$ can be written as perturbation of a rotation. Employing a rotating coordinate system to nullify the rotation of $z$, and writing the transformed coordinates as $\zedx$ we show that, under an appropriate change of time scale, the law of $\zedx$ converges to the law of diffusion process as $\eps \to 0$. This result is useful because {\it for small $\eps$  it provides an approximate two dimensional description of countably infinite modes that the delay equation possesses.} 

{\bf The result is summarized in theorem \ref{thm:mainthmExpErgNoimainres}}. For vector-valued DDE without proofs, and an illustration of the usefulness of these results using numerical simulations, see \cite{PRE}.

\subsection{Related work}
Systems with small noise perturbations are studied in \cite{Khas2, Blankpap, Duke}. They consider systems of the form
$\frac{d}{d\tau}\tilde{x}^{\eps}(\tau)={\eps}F(\tilde{x}^{\eps}(\tau),\xi(\tau))+\eps^2G(\tilde{x}^{\eps}(\tau),\xi(\tau))$
with $F$ such that for each fixed $\tilde{x}$, $\expt[F(\tilde{x},\xi)]=0$ where expectation is with respect to the invariant measure of the noise $\xi$. On changing the time sacle in the above equation: $t=\eps^2\tau$, $x^{\eps}(t):=\tilde{x}^{\eps}(t/\eps^2)$, $\xi^{\eps}(t):=\xi(t/\eps^2)$, we have
$\frac{dx^{\eps}(t)}{dt}=\frac{1}{\eps}F(x^{\eps}(t),\xi^{\eps}(t))+G(x^{\eps}(t),\xi^{\eps}(t)).$
It is shown in \cite{Khas2, Blankpap, Duke} (using different assumptions) that the law of $x^\eps$ converges weakly to that of a diffusion process as $\eps \to 0$. Analogous results for DDE are in \cite{YinRama}. It considers
\begin{equation}\label{eq:YinRamaEqmain}
\dot{x}^\eps(t)=\frac{1}{\eps}b(x^\eps(t),x^\eps(t-r),\xi^\eps(t))\,+\,a(x^\eps(t),x^\eps(t-r),\xi^\eps(t))
\end{equation}
with the assumptions that $\expt[b(x,x_r,\xi(t))]=0$, $\expt[a(x,x_r,\xi(t))]=\bar{a}(x,x_r)$ $\forall\, x,x_r,$ and
$$\frac{1}{(T_2-T_1)}\int_{T_1}^{T_2}\expt\big[\sum_j\frac{\partial b_i(x,x_r,\xi(t))}{\partial x_j}\,\,b_j(x,x_r,\xi(T_1))\big]\,dt \to \bar{b}(x,x_r) \quad \text{ as } T_1,T_2, T_2-T_1 \to \infty,$$
$$\frac{1}{(T_2-T_1)}\int_{T_1}^{T_2}\expt[b_i(x,x_r,\xi(t))\,\,b_j(x,x_r,\xi(T_1))]\,dt \to \frac12 S_{ij}(x,x_r) \quad \text{ as }  T_1,T_2,T_2-T_1 \to \infty,$$
$$\exists \Phi(x,x_r) \text{ such that } \frac12(S_{ij}+S_{ij}^T)=\Phi\Phi^T.$$
Then \cite{YinRama} shows that, as $\eps \to 0$, the law of $x^\eps$ converges weakly to that of a stochastic DDE given by
$$dx(t)=\big[\bar{a}(x(t),x(t-r)) \,+\,\bar{b}(x(t),x(t-r))\big]\,dt \,\,+\,\,\Phi(x(t),x(t-r))\,dW(t).$$
Note that \eqref{eq:YinRamaEqmain} is a time-rescaled version of 
$\dot{x}(t)= \eps b(x(t),x(t-\frac{r}{\eps^2}),\xi(t))+\eps^2 a(x(t),x(t-\frac{r}{\eps^2}),\xi(t)).$
If you compare this with \eqref{eq:detDDE_pert_gennoise_intro}: here the delay is not a constant as $\eps$ varies, whereas in \eqref{eq:detDDE_pert_gennoise_intro} delay is a constant. Alternatively, in \eqref{eq:YinRamaEqmain} the delay is fixed delay $r$, whereas a time-rescaled version of \eqref{eq:detDDE_pert_gennoise_intro} would have vanishing delay $\eps^2 r$. Whereas \cite{YinRama} obtains an SDDE in the limit, we would obtain an SDE without delay.

In \cite{pap_var} the effect of noise on evolution equations on Banach spaces is considered, and \cite{PapKoh75} extends it to systems with fast and slow components. \cite{PapKoh75} considers 
\begin{align}\label{eq:PapKohcons}
\frac{dy^{\eps}(t)}{dt}=\frac{1}{\eps}By^{\eps}(t)+A(t/\eps)y^{\eps}(t),
\end{align}
with the following assumptions: (i) the operator $B$ (deterministic) generates a contraction semigroup which is denoted by $e^{tB}$, (ii) $B$ is such that $e^{tB}\to \pih$ as $t\uparrow \infty$, where $\pih$ is the projection onto the kernel of $B$, (iii) $\exists\,C,\,\gamma>0$ such that $||(e^{tB}-\pih)f||\leq Ce^{-\gamma t}||f||$. Under the assumptions (i) and (ii), we have that $e^{tB}\pih=\pih e^{tB}=\pih$ and $\pih Bf=B\pih f=0$.  Define the operator $\bar{A}$ by 
$\bar{A}=\lim_{T\uparrow \infty}\frac{1}{T}\int_{t_0}^{t_0+T}\expt[A(s)]ds.$
Write the solution of the equation \eqref{eq:PapKohcons} as $y^{\eps}(t)=U^{\eps}(t,0)y^{\eps}(0)$. \cite{PapKoh75} is concerned with the asymptotic behaviour of $U^{\eps}(t,0)$ as $\eps \downarrow 0$. Under some assumptions on $U^{\eps}(t,0)$ \cite{PapKoh75} states \\
\emph{Theorem 3.1 in \cite{PapKoh75}}: $\qquad$ For $0\leq t\leq T,$ $\lim_{\eps \downarrow 0}\expt[U^{\eps}(t,0)\pih f]=e^{t\,\pih \bar{A}\pih}f.$ \\
\emph{Theorem 3.2 in \cite{PapKoh75}: $\qquad$ Suppose $\pih\,\bar{A}\,\pih \equiv 0$. Then, for $0\leq t\leq T,$ $\lim_{\eps \downarrow 0}\expt[U^{\eps}(t/\eps,0)\pih f]=e^{t\bar{V}}\pih f,$
where $\bar{V}=\lim_{T\uparrow \infty}\frac{1}{T}\int_{t_0}^{t_0+T}\int_{t_0}^s\expt[\pih A(s)(e^{B(s-u)}-\pih)A(u)\pih]\,du \,ds.$}\\
The above result theorem 3.2 is in fact not correct. 
When the fast component is present the theorem  gives only the critical(slow)-stable(fast) interaction, but not the critical(slow)-critical(slow) interaction. In fact doing the computations in \cite{PapKoh75} carefully shows that the correct result is $$\bar{V}=\lim_{T\uparrow \infty}\frac{1}{T}\int_{t_0}^{t_0+T}\int_{t_0}^s\expt[\pih A(s)e^{B(s-u)}A(u)\pih]\,du \,ds.$$
The delay equation \eqref{eq:detDDE_pert_gennoise_intro} could be put in the framework of \cite{PapKoh75}. However it is difficult to satisfy all the assumptions and so we choose the easier route of using the martingale problem technique. The case of $\int_{\nstsp}\sigma(\zeta)\bar{\nu}(d\zeta)\neq 0$  corresponds to the case of  theorem 3.1 in \cite{PapKoh75} and $\int_{\nstsp}\sigma(\zeta)\bar{\nu}(d\zeta)= 0$  corresponds to the case of  theorem 3.2 in \cite{PapKoh75}.

\cite{Tsark1} considers equations of the form \eqref{eq:detDDE_pert_gennoise_intro} with $\int_{\nstsp}\sigma(\zeta)\bar{\nu}(d\zeta)\neq 0$ and a different $\eps$ scaling; for example: $dx(t)=L_0(\pj_t x)dt+\eps G(\pj_t x)dt + \eps \sigma(\xi(t)) F(\pj_t x)dt$.  Let $z$ be the coordinates of $P$-projection of $\pj_tx$ with respect to basis $\Phi$. Let $\zedx$ be the transformed process obtained after nullifying the rotation of $z$. \cite{Tsark1} shows that the probability law of $\zedx(t/\eps)$ converges to that of a deterministic ODE and that the norm of $Q$-projection decays exponentially fast. If the zero fixed point of the limit ODE is exponentially stable, then it is proven that $x$ is also exponentially stable in the moments.

\cite{NavWihs, NavWihs2} considers equations of the form $\dot{x}(t)=L_0(\pj_t x)+\eps\sigma(\xi(t))L_1(\pj_t x)$ with $\sigma$ a mean zero function of the noise process $\xi$ and $L_1$ is a bounded linear operator on $\C$. Define the exponential growth rate $\lambda^\eps := \limsup_{t\to \infty}\frac{1}{t}\log|x^\eps(t)|$ and expand it as
$\lambda^\eps=\lambda_0+\eps \lambda_1+\eps^2\lambda_2+\ldots$.
Using perturbation methods and Furstenberg-Khasminskii representation, \cite{NavWihs, NavWihs2} show that $\lambda_0=\lambda_1=0$ and give explicit expression for $\lambda_2$.

\subsection{Organization of this paper} In sections  \ref{sec:unpertprob} we collect the results on spectral properties of linear DDE that would be useful to us. In section \ref{sec:pertprob} we arrive at coupled equations for the evolution of projections of $\pj_tx$ of the system \eqref{eq:detDDE_pert_gennoise_intro} on to the critical and stable eigenspaces. The sections \ref{sec:unpertprob} and \ref{sec:pertprob} are just recalling the set-up from \cite{NavWihs2} which draws from \cite{Halebook} and \cite{Diekmanbook}. In Section \ref{sec:convgTruncProb} we prove the weak convergence result following \cite{Duke}, using the technique of martingale problem. For the brevity of notation in section \ref{sec:convgTruncProb} we work with $G=G_q=0$. In section \ref{sec:GGqExpErgNoi}  we consider the effect of $G$ and $G_q$. For convenience of the reader, the final result is summarized in section \ref{sec:MainResultExpErgNoi}.

\section{The unperturbed deterministic system}\label{sec:unpertprob}
Here we are just recalling the set-up from \cite{NavWihs2} which draws from \cite{Halebook} and \cite{Diekmanbook}.

The solution of \eqref{eq:detDDE} gives rise to the strongly continuous semigroup $T(t):\C\to \C, \, t\geq 0$, defined by  
$T(t)\pj_0x=\pj_t x$. The generator $\gen$ of the semigroup is given by 
\begin{equation}\label{eq:generator_def_unpert}
\gen \varphi=\frac{d}{d\theta}\varphi, \qquad dom(\gen)=\dom(\gen)=\{\varphi \in \C^1|\varphi'(0)=L_0\varphi\}
\end{equation}
($\C^1$ is the linear space of continuously differentiable functions on $[-r, 0]$, and $'= \frac{d}{d\theta}$).  The equation \eqref{eq:detDDE} with the initial condition $\varphi$ in $\dom(\gen)$, is equivalent to the abstract differential equation
\begin{align}\label{eq:abstDE_unpert}
\frac{d}{dt}\pj_tx=\gen \,\pj_tx, \qquad t\geq 0; \quad \pj_0x=\icond \in \dom(\gen), 
\end{align}
where the differentiation with respect to $t$ is taken in the sense of the sup-norm in $\C$.

The state space $\C$ splits in the form $\C=P\oplus Q$ where $P=\text{span}_{\R}\{\Phi_1,\,\Phi_2\}$ where (recall $\pm i\om_c$ are the critical eigenvalues)
$$\Phi_1(\theta)=\cos(\om_c\theta), \qquad \Phi_2(\theta)=\sin(\om_c\theta), \qquad \theta\in[-r,0].$$
Write $\Phi=[\Phi_1,\,\Phi_2]$. Using the identity $\cos(\omega_c(t+\cdot))=\cos(\omega_ct)\cos(\omega_c\cdot)-\sin(\omega_ct)\sin(\omega_c\cdot)$ and the linearity of $L_0$, it can be shown that
\begin{align}\label{eq:TPhi_eq_PhieB}
T(t)\Phi(\cdot)=\Phi(\cdot)e^{Bt}, \qquad B=\left[\begin{array}{cc}0 & \omega_c \\ -\omega_c & 0 \end{array}\right],
\end{align}
with the derivative
\begin{equation}\label{eq:APhi=PhiB}
\gen\Phi(\cdot)=\frac{d}{dt}T(t)\Phi(\cdot)|_{t=0}=\Phi(\cdot)B.
\end{equation}

Let $\pi$ denote the projection of $\C$ onto $P$ along $Q$, i.e. $\pi:\C\to P$ with $\pi^2=\pi$ and $\pi(\eta)=0$ for $\eta\in Q$.
The operator $\pi$ can be represented using a bilinear form $\la \cdot,\cdot\ra$ on $C([-r,0],\R)\times C([0,r],\R)$ given by
\begin{equation}\label{eq:bilinform}
\la \phi,\psi \ra := \phi(0)\psi(0)-L_0(\int_0^{\cdot}\phi(u)\psi(u-\cdot)du)
\end{equation}
and functions $\Psi(\cdot)=\left[\begin{array}{c}\Psi_1(\cdot) \\ \Psi_2(\cdot) \end{array}\right]$, where $\Psi_i$ are linear combinations of $\cos(\om \cdot)$ and $\sin(\om \cdot)$ and are such that $\la \Phi_i,\Psi_j\ra=\delta_{ij}$. We have for the projection $\pi: \C \to P$, \begin{align}\label{eq:projoper_intermsof_bilform} \pi(\eta)=\Phi\la \eta,\Psi \ra = \la \eta, \Psi_1\ra \Phi_1 + \la \eta, \Psi_2\ra \Phi_2, \end{align}
and $Q=ker(\pi)=\{\eta \in \C | \pi(\eta)=0\}.$
There exists positive constants $\kappa$ and $K$ such that  
\begin{align}\label{eq:expdecesti_befext}
||T(t)\phi||\leq Ke^{-\kappa t}||\phi||, \qquad \quad \forall\, \phi \in Q. 
\end{align}

Write the solution to \eqref{eq:detDDE} (with initial condition in $dom(\gen)$) as $\pj_tx=\pi\pj_tx+(I-\pi)\pj_tx$ and define $z, y$ by $\Phi z(t)=\pi\pj_t x$ and $y_t:=(I-\pi)\pj_tx$. Here $z(t)$ is $\R^2$-valued. Then, using \eqref{eq:APhi=PhiB}, equation \eqref{eq:abstDE_unpert} can be replaced by the following system of equations
\begin{equation}\label{eq:proj_eqns_unpert}
\dot{z}(t)=Bz(t), \qquad \quad \frac{d}{dt}y_t=\gen y_t
\end{equation} 
with initial values $z(0)$ and $y_0(\cdot)$ given by $\pj_0x(\cdot)=\icond(\cdot)=\Phi(\cdot)z(0)+y_0(\cdot)$.

From \eqref{eq:proj_eqns_unpert} it can be noted that $z$ oscillate with frequency $\om_c$ and constant amplitude. From \eqref{eq:expdecesti_befext} it can be noted that $||y_t||$ decays exponentially fast.

We aim at investigating the noise perturbed system \eqref{eq:detDDE_pert_gennoise_intro}, that is, comparing the unperturbed system with the perturbed one within the same framework. But if noise is present the boundary condition \eqref{eq:generator_def_unpert}, in general, cannot be satisfied: for example, when $G=G_q=0$, the equation $\frac{d}{d\theta}\icond(0) = L_0\icond + \eps F(\icond)\sigma(\xi_0(\omega))$ cannot hold for all chance elements $\omega$. Therefore, jumps must be allowed as result
of differentiation. In other words, we need to extend the space $\C$ together with the operators $\gen$ and $T(t)$. See chapter 6, 7 of \cite{Halebook}, and \cite{Diekmanbook} for the extension. See \cite{NavWihs2} for a summary of the results that follow \cite{ChowParet}.

Let $\Ch=span_{\R}\Ind \oplus L^{\infty}$ where $L^{\infty}=L^{\infty}([-r,0],\mathcal{B}([-r,0]),Leb;\R)$. The norm on $\Ch$ is $||\gamma \Ind+\phi||=|\gamma|+||\phi||_{ess\,sup}.$ Let the extension of the semigroup $T$ to $\Ch$ be denoted by $\Th$.
The generator $\gen$ is extended to 
\begin{equation}\label{eq:Ahat}
\Ah\phi=(L_0\phi-\phi'(0))\Ind + \frac{d}{d\theta}\phi, \qquad \phi\in Lip, \qquad '\equiv\frac{d}{d\theta}
\end{equation}
where $Lip=\cup_{\alpha \in \R}Lip(\alpha)$ and $Lip(\alpha)$ is the equivalence class in $L^{\infty}$ which contain at least one Lipschitz-continuous function $\phi$ with $\phi(0)=\alpha$. 
If $\phi$ is not continuous at 0, we represent the equivalence class $[\phi]$ by a function $\phi$ with $\phi(0)=0$. Note that at any rate $(\Ah\phi)(0) = L_0\phi$. 

$\Ah$ has the same spectrum as $\gen$ and the decomposition of $\C$ is lifted to the decomposition of $\Ch$ (see \cite{Diekmanbook}, page 100) with help of the extension of $\pi$ to the projection $\pih : \Ch \to P_{\Lambda}$. Let $\Psiz=\Psi(0)$ and $\Psiz_i=\Psi_i(0)$. Explicitly, we only need
$$\pih(\gamma \Ind)=\gamma \Phi \Psiz=\gamma(\Psiz_1\Phi_1+\Psiz_2\Phi_2).$$
We have $\Ch=P\oplus \Qh$ where $\Qh=ker(\pih)$. The spaces $P$ and $\Qh$ are $\Ah$-invariant and the commutativity property $\pih\Ah=\Ah \pih$ holds on $\dom(\Ah)=Lip$.
There exists positive constants $\kappa$ and $K$ such that  
\begin{align}\label{eq:expdecesti_befext_aft}
||T(t)\phi||\leq Ke^{-\kappa t}||\phi||, \qquad \quad \forall\, \phi \in \Qh. 
\end{align}


\section{The randomly perturbed system}\label{sec:pertprob}
We are now prepared to study the randomly perturbed system \eqref{eq:detDDE_pert_gennoise_intro}. To keep the notation simple we first deal with the case $G=G_q=0$. The effect of $G$ and $G_q$ are studied in section \ref{sec:GGqExpErgNoi}. So we consider
\begin{align}\label{eq:detDDE_pert_gennoise_intro_GGqZero}
\begin{cases}dx(t)=L_0(\pj_t x)dt+\eps \sigma(\xi(t)) F(\pj_t x)dt, \quad t\geq 0, \\
\pj_0x =\icond \in Lip,
\end{cases}
\end{align}
The equation \eqref{eq:detDDE_pert_gennoise_intro_GGqZero} is equivalent to the abstract differential equation
\begin{equation}\label{eq:estabcorrespwithPapKoh}
\frac{d}{dt}\pj_tx\,\,=\,\,\Ah\,\pj_tx\,\,+\,\,\eps \sigma(\xi_t) F(\pj_tx) \Ind, \qquad \pj_0x=\icond \in Lip.
\end{equation}
Writing $\pj_t x=\Phi z(t)+y_t$ with $\Phi z(t)=\pih \pj_tx$ and $y_t:=(I-\pih)\pj_tx$, and using the facts that $\Ah(\Phi z)=\Phi Bz$, $\Ah$ commutes with $\pih$ on $\dom(\Ah)$ and $\pih\Ind=\Phi\Psiz$,  we have
\begin{eqnarray}\label{eq:proj_eqns_pert_crit}
\dot{z}(t) &=& Bz(t)+\eps \sigma(\xi_t) F(\Phi z(t)+y_t)\Psiz, \\ \label{eq:proj_eqns_pert_stable}
\frac{d}{dt}y_t &=& \Ah y_t+ \eps \sigma(\xi_t) F(\Phi z(t)+y_t)(I-\pih)\Ind,
\end{eqnarray}
with initial conditions $z(0)\in \R^2$ and $y_0\in \Qh \cap Lip$ such that $\pj_0x=\icond = \Phi z(0)+y_0$.

From \eqref{eq:proj_eqns_pert_crit} it can be noted that dynamics of $z$ is small perturbation of a rotation with frequency $\om_c$. To study the effect of perturbation itself, we need to employ a coordinate system which nullifies the rotation of $z$. Further, the noise perturbations take order $O(1/\eps^2)$ time to significantly affect the dynamics. So, we employ the following transformation:
$$\zed_t=e^{-Bt/\eps^2}z(t/\eps^2), \qquad \quad \xi^\eps_t=\xi_{t/\eps^2}, \qquad \quad y^{\eps}_t:=y_{t/\eps^2}, \qquad \quad \teps_t:=t/\eps^2,$$  
and write the evolution equations as
\begin{eqnarray}\label{eq:proj_eqns_pert_crit_alt}
\frac{d}{dt}{\zed_t} &=& \frac{1}{\eps} \sigma(\xi_t^\eps) F(\Phi e^{B\teps_t}\zed_t+y_t^\eps)e^{-B\teps_t}\Psiz, \\ \label{eq:proj_eqns_pert_stable_alt}
\frac{d}{dt}y_t^\eps &=& \frac{1}{\eps^2}\Ah y_t^\eps+ \frac{1}{\eps} \sigma(\xi_t^\eps) F(\Phi e^{B\teps_t}\zed_t+y_t^\eps)(I-\pih)\Ind,
\end{eqnarray}
The process $(\zed_t,\teps_t, y^\eps_t,\xi^\eps_t)$ is a Markov process on $\Sp:=\R^2\times (\R^+ \cup \{0\})\times (\Qhl \cap Lip) \times \nstsp$.

Our goal is to show that the probability law of $\zed$ converges to the probability law of a two-dimensional diffusion process. For the proof we use the technique of martingale problem. The procedure we employ is as follows. Consider the following truncated process:
\begin{eqnarray}\label{eq:proj_eqns_pert_crit_alt_trunc}
\frac{d}{dt}{\zedaux{n}_t} &=& \frac{\trunc_n(\zedaux{n}_t)}{\eps} \sigma(\xi_t^\eps) F(\Phi e^{B\teps_t}\zedaux{n}_t+y_t^{\eps,n})e^{-B\teps_t}\Psiz, \\ \label{eq:proj_eqns_pert_stable_alt_trunc}
\frac{d}{dt}y_t^{\eps,n} &=& \frac{1}{\eps^2}\Ah y_t^{\eps,n}+ \frac{\trunc_n(\zedaux{n}_t)}{\eps} \sigma(\xi_t^\eps) F(\Phi e^{B\teps_t}\zedaux{n}_t+y_t^{\eps,n})(I-\pih)\Ind,
\end{eqnarray}
with $\trunc_n:\R^2\to \R$ a smooth function given by
$$\trunc_n(\zedx)=\begin{cases}1 \qquad \text{ for } ||\zedx||_2\leq n,\\ 0 \qquad \text{ for } ||\zedx||_2\geq n+1.\end{cases}$$
Define the stopping time $\stopt_n := \inf\{t\geq 0\,:\,||\zed_t||_2\geq n\}$. Then the law of $\zed$ agrees with law of $\zedaux{n}$ until $\stopt_n$. We identify some drift $b^{(n)}$ and diffusion coefficient $a^{(n)}$, and show that as $\eps \to 0$, the law of the truncated processes $\zedaux{n}$ converge to the unique solution of martingale problem with diffusion and drift coefficients $(a^{(n)},b^{(n)})$. We identify $(a,b)$ so that $a\equiv a^{(n)}$, $b\equiv b^{(n)}$ on $\{\zedx\in \R^2\,:\,||\zedx||_2\leq n\}$ and show that there exists unique solution for the martingale problem of $(a,b)$. By corollary 10.1.2 and lemma 11.1.1 of \cite{MDP} we then have that the law of $\zed$ converges as $\eps\to 0$ to the law of diffusion process with diffusion and drift coefficients $(a,b)$.

In the next section we show that the law of truncated process $\zedaux{n}$ converges as $\eps\to 0$ to that of a diffusion process. The following notation would be used.
\begin{itemize}
\item For $f:\Ch \to \R$ and $\eta \in \Ch$, $\zeta \in \Ch$, let $(\zeta.\nabla) f(\eta)$ denote the Frechet derivative of $f$ at $\eta$ in the direction $\zeta$, i.e.
$$(\zeta.\nabla) f(\eta)\,:=\,\lim_{\delta\to 0}\frac{1}{\delta}\big(f(\eta+\delta \zeta)-f(\eta)\big).$$ 
\item For $f:\R^2\to \R$ such that $f\in C^1(\R^2;\R)$, and $v\in \R^2$, $$(v.\nablaz)f(z)=\lim_{\delta\to 0}\frac{1}{\delta}\big(f(z+\delta v)-f(z)\big).$$
\item For $f: (\Qh \cap Lip)\to \R$ differentiable, and $\tilde{y}\in \Qh \cap Lip$,  $$(\tilde{y}.\nablay)f(y)=\lim_{\delta\to 0}\frac{1}{\delta}\big(f(y+\delta \tilde{y})-f(y)\big).$$
\end{itemize}


\section{Convergence of the law of $\{\zedaux{n}\}_{\eps>0}$}\label{sec:convgTruncProb}
Let $N_*\in \mathbb{N}$ such that for the initial condition $\pj_0x=\icond$ both $||z(0)||_2=||\zed_0||_2<N_*$  and $||(I-\pi)\icond||=||y_0^\eps||<N_*$ holds.
We consider only $n\geq N_*$.

From \eqref{eq:proj_eqns_pert_crit_alt_trunc} it can be seen that 
\begin{align}\label{eq:ztruncbound}
||\zedaux{n}_t||\,\leq\, n+1, \quad t\geq 0.
\end{align}
Employing this fact in the variation-of-constants formula
$$y_t^{\eps,n}=\Th(\frac{t-s}{\eps^2})y_0^{\eps,n}+\int_0^t\Th(\frac{t-s}{\eps^2})(1-\pih)\Ind \sigma(\xi_s^\eps) F(\Phi e^{B\teps_s}\zedaux{n}_s+y_s^{\eps,n})ds,$$
using the exponential decay \eqref{eq:expdecesti_befext_aft} and Gronwall inequality, it can be shown that 
\begin{align}\label{eq:ytruncbound}
||y_t^{\eps,n}||\,\,\leq\,\, K||y_0||e^{-\kappa t/\eps^2}+\eps C(n+||y_0||)\,\,\leq\,\,Cn, \qquad t\geq 0.
\end{align}
Thus, once $n$ is fixed, $\zedaux{n}_t$ and $y^{\eps,n}_t$ are bounded with the bound independent of $\eps$.

The truncated process $(\zedaux{n}_t,\teps_t, y^{\eps,n}_t,\xi^\eps_t)$ is a Markov process on $$\Sp_n:=(\R^2)_{n}\times (\R^+ \cup \{0\})\times (\Qhl \cap Lip)_n \times \nstsp$$ where $$(\R^2)_{n}:=\{\zedx \in \R^2\,:\,||\zedx||_2\leq n+1\}, \qquad (\Qhl \cap Lip)_n=\{y \in \Qhl \cap Lip\,\,:\,\,||y||\leq Cn\}$$ where $C$ is from \eqref{eq:ytruncbound}.

The infinitesimal generator $\igen^\eps$ of $(\zedaux{n}_t,\teps_t, y^{\eps,n}_t,\xi^\eps_t)$ is defined as follows. Let
\begin{align*}
\igen_{0}&=\noisegent + (\Ah y).\nablay + \frac{\partial}{\partial \tau},\\
\igen_{1,1}&=\trunc_n(\zedx)\sigma(\xi) F(\Phi e^{B\tau}\zedx+y)(e^{-B\tau}\Psiz).\nablaz, \\
\igen_{1,1}&=\trunc_n(\zedx)\sigma(\xi) F(\Phi e^{B\tau}\zedx+y)((I-\pih)\Ind).\nablay
\end{align*}
Then, the infinitesimal generator $\igen^\eps$ is
\begin{equation}\label{eq:definfgen_Lteps}
\igen^{\eps}\,\,=\,\,\frac{1}{\eps^2}\igen_0+\frac{1}{\eps}\igen_{1}\,\,=\,\,\frac{1}{\eps^2}\igen_0+\frac{1}{\eps}(\igen_{1,1} +\igen_{1,2}).
\end{equation}

\begin{thm}\label{thm:apprxofgenacttestfunc}
Let $\igen^\eps$ be the generator of $(\zedaux{n}_t,\teps_t, y^{\eps,n}_t,\xi^\eps_t)$ process and let $\igen^{\dag}_n$ be as in equation \eqref{eq:def:avggendef} (more explicitly written in remark \ref{rmk:drifdiffcoeffLndag}). Then for any {$g \in C^3(\R;\R)$} with bounded derivatives, there exists functions (``correctors'') $\psi^{(k)}_g(\zedx,\tau,y,\xi)$, k=1,2, bounded on $\Sp_n$ such that
\begin{align}\label{eq:thm:apprxofgenacttestfunc}
\igen^\eps\left( g + \eps \psi^{(1)}_g + \eps^2 \psi^{(2)}_g\right) \,\,=\,\,\igen^{\dag}_ng + \eps\igen_1\psi^{(2)}_g
\end{align}
The correctors  $\psi^{(k)}_g$ are given in equations \eqref{eq:correctorpsi1eq}--\eqref{eq:correctorpsi2eq_psi21psi22}. Further, $\igen_1\psi^{(2)}_g$ is bounded on $\Sp_n$. 
\end{thm}
\begin{proof}
First we set-up some notation. For functions of $(\zedx,\tau,y,\xi)$, define $$(e^{t\igen_0}f)(\zedx,\tau,y,\xi):=\int_{\nstsp}f(\zedx,\tau+t,\Th(t)y,\zeta)\nu(t,\xi,d\zeta).$$
For functions of $(\zedx,\tau,y,\xi)$, define operator $\xiavg$ by $$[\xiavg f](\zedx,\tau,y)=\int_{\nstsp}f(\zedx,\tau,y,\xi)\bar{\nu}(d\xi).$$
For functions of $(\zedx,\tau,y)$, define operator $\yavg$ by $$[\yavg f](\zedx,\tau)=f(\zedx,\tau,0).$$
For functions of $(\zedx,\tau)$, define operator $\tauavg$ by $$[\tauavg f](\zedx)=\frac{1}{2\pi/\om_c}\int_0^{2\pi/\om_c}f(\zedx,\tau)d\tau.$$

Now, expanding LHS of \eqref{eq:thm:apprxofgenacttestfunc}, and noting that $\igen_0g=0$ and $\igen_{1,2}g=0$, we get
\begin{align}\label{eq:expand_pow_of_eps_polargenonperttest}
\frac{1}{\eps}\bigg(\igen_0\psi^{(1)}_g \,+\,\igen_{1,1}g\bigg)\,+\,\bigg(\igen_0\psi^{(2)}_g \,+\,\igen_{1}\psi^{(1)}_g\bigg)\,+\,\eps\bigg(\igen_1\psi^{(2)}_g \bigg).
\end{align}
The function 
\begin{align}\label{eq:correctorpsi1eq}
\psi^{(1)}_g=\int_0^\infty e^{s\igen_0}\igen_{1,1}g\,ds
\end{align}
solves $\igen_0\psi^{(1)}_g \,+\,\igen_{1,1}g=0$. The integral from $0$ to $\infty$ is well-defined because (i) by mean-zero and bounded nature of $\sigma$ and assumption \ref{assmp:assmp_on_noise} we have the exponential decy: 
$$\bigg|\int_{\nstsp}\sigma(\zeta)\nu(t,\xi,d\zeta)\bigg| = \bigg|\int_{\nstsp}\sigma(\zeta)(\nu(t,\xi,d\zeta)-\bar{\nu}(d\zeta))\bigg| \leq |\sigma|\sup_{\zeta\in \nstsp}\int_{\nstsp}|\nu(t,\xi,d\zeta)-\bar{\nu}(d\zeta)| \leq |\sigma|c_1e^{-c_2t},$$
and (ii) $\igen_{1,1}g$ is bounded on $\Sp_n$ due to atmost linear growth of $F$.

Now we try to find $\psi^{(2)}_g$. However, note that $\igen_{1}\psi^{(1)}_g$ is not mean-zero. In order to have an exponential decay, we choose $\psi^{(2,a)}_g=\int_0^\infty e^{s\igen_0}\left(\igen_{1}\psi^{(1)}_g-\xiavg(\igen_{1}\psi^{(1)}_g)\right)ds$. Then $\psi^{(2,a)}_g$ solves $\igen_0\psi^{(2,a)}_g + \igen_{1}\psi^{(1)}_g-\xiavg(\igen_{1}\psi^{(1)}_g) = 0$. We then choose $\psi^{(2,b)}_g=\int_0^\infty e^{s\igen_0}\left(\Xi(\igen_{1}\psi^{(1)}_g)-\yavg\xiavg(\igen_{1}\psi^{(1)}_g)\right)ds$. This time, the exponential decay would be provided by bounded derivatives of $F$ and the decay at \eqref{eq:expdecesti_befext_aft}. Then $\psi^{(2,b)}_g$ solves $\igen_0\psi^{(2,b)}_g + \xiavg(\igen_{1}\psi^{(1)}_g)-\yavg\xiavg(\igen_{1}\psi^{(1)}_g) = 0$.  We then choose $\psi^{(2,c)}_g(\zedx,\tau)=-\int_0^{\tau} \left(\yavg\Xi(\igen_{1}\psi^{(1)}_g)-\tauavg\yavg\xiavg(\igen_{1}\psi^{(1)}_g)\right)|_{(\zedx,s)}ds$. The integrand here  is bounded and is periodic in $\tau$ with average zero and so $\psi^{(2,c)}_g$ is bounded. Now, $\psi^{(2,c)}_g$ solves $\igen_0\psi^{(2,c)}_g + \yavg\xiavg(\igen_{1}\psi^{(1)}_g)-\tauavg\yavg\xiavg(\igen_{1}\psi^{(1)}_g) = 0$. Let 
\begin{align}\label{eq:correctorpsi2eq_psi21psi22}
\psi^{(2)}_g \,\,:=\,\,\psi^{(2,a)}_g +\psi^{(2,b)}_g +\psi^{(2,c)}_g.
\end{align}
Then $\psi^{(2)}_g$ solves $\igen_0\psi^{(2)}_g + \igen_{1}\psi^{(1)}_g-\tauavg\yavg\xiavg(\igen_{1}\psi^{(1)}_g) = 0$.

Define $\igen^\dag_n$ by 
\begin{align}\label{eq:def:avggendef}
\igen^\dag_n g \,\,=\,\,\tauavg\yavg\xiavg(\igen_{1}\psi^{(1)}_g).
\end{align}
Collecting all the above, we have \eqref{eq:thm:apprxofgenacttestfunc}. Bounded derivatives of $F$ and $g$ ensure that the correctors are bounded on $\Sp_n$ and that $\igen_1\psi^{(2)}_g$ is bounded on $\Sp_n$.
\end{proof}

\begin{rmk}\label{rmk:drifdiffcoeffLndag}
The generator $\igen^\dag_n$ can be explicitly written in the following form:
\begin{align}
\igen^\dag_n = \sum_{i=1}^2b^{(n)}_i(\zedx)\frac{\partial}{\partial \zedx_i} + \frac12\sum_{i,j=1}^2a^{(n)}_{ij}(\zedx)\frac{\partial^2}{\partial \zedx_i\partial \zedx_j}
\end{align}
where $a^{(n)}, b^{(n)}$ are as written as follows in terms of the auto-correlation  function $R$ of the noise:
\begin{align}\label{def:autocorrNoiseExpErgNoi}
R(t):=\int_{\nstsp}\sigma(\xi)\left(\int_{\nstsp}\sigma(\zeta)\nu(t,\xi,d\zeta)\right)\bar{\nu}(d\xi).
\end{align}
Let 
\begin{align}\label{eq:diffcoeffExpErgNoi}
a_{ij}(\zedx) &= \frac{1}{2\pi/\om_c}\int_0^{2\pi/\om_c}d\tau\int_0^\infty R(s) F(\Phi e^{\tau B}\zedx)F(\Phi e^{(\tau+s)B}\zedx)\big(e^{-\tau B}\Psiz\big)_i\big(e^{-(\tau+s)B}\Psiz\big)_j\,ds,
\end{align}
\begin{align*}
b_i(\zedx) \,\,=\,\, b^{F}_i(\zedx) \,\,=\,\, b^{F,P}_i(\zedx) + b^{F,Q}_i(\zedx),
\end{align*}
\begin{align}\label{eq:driftcoeffbFPExpErgNoi}
b^{F,P}_i(\zedx) &= \frac{1}{2\pi/\om_c}\int_0^{2\pi/\om_c}d\tau\int_0^\infty R(s) F(\Phi e^{\tau B}\zedx)\bigg(\big(\Phi e^{sB}\Psiz \big).\nabla\bigg)F(\Phi e^{(\tau+s)B}\zedx)\,\big(e^{-(\tau+s)B}\Psiz\big)_i\,ds,\\ \label{eq:driftcoeffbFQExpErgNoi}
b^{F,Q}_i(\zedx) &= \frac{1}{2\pi/\om_c}\int_0^{2\pi/\om_c}d\tau\int_0^\infty R(s) F(\Phi e^{\tau B}\zedx)\bigg(\big(\Th(s)(I-\pih)\Ind \big).\nabla\bigg)F(\Phi e^{(\tau+s)B}\zedx)\,\big(e^{-(\tau+s)B}\Psiz\big)_i\,ds.
\end{align}
Then, 
\begin{align*}
a^{(n)}_{ij} = \trunc_n^2a_{ij}, \qquad b^{(n)}_i = \trunc_n^2b^{F}_i + b^{(n),F,\trunc}_i,
\end{align*}
where
\begin{align*}
b^{(n),F,\trunc}_i(\zedx) &= \frac{\trunc_n(\zedx)}{2\pi/\om_c}\int_0^{2\pi/\om_c}d\tau\int_0^\infty R(s) F(\Phi e^{\tau B}\zedx)F(\Phi e^{(\tau+s)B}\zedx)\bigg(\big(e^{-\tau B}\Psiz \big).\nablaz\bigg)\trunc_n(\zedx)\,\big(e^{-(\tau+s)B}\Psiz\big)_i\,ds.\\
\end{align*} 
Note that $(a^{(n)},b^{(n)})$ agree with $(a,b)$ on the set $\{\zedx\in \R^2\,:\,||\zedx||_2\leq n\}$. $\hfill \qed$
\end{rmk}

\vspace{10pt}

We assume the $\Sp$-valued processes $(\zedaux{n}_t,\teps_t, y^{\eps,n}_t,\xi^\eps_t)$ are defined on probability triples $(\Om^{\eps,n},\fil^{\eps,n}_t,\mbbP^{\eps,n})$. Let $\mbbP^{\eps}$ be probability law of $(\zed_t,\teps_t, y^{\eps}_t,\xi^\eps_t)$, i.e. without truncation. Also define the following canonical set-up.
\begin{defn}\label{def:canonicalsetupreduced}
Define $\Omega^\dag:=C([0,\infty),\R^2)$ equipped with the metric
$$D(\om,\om')=\sum_{n=1}^\infty \frac{1}{2^n}\frac{\sup_{t\in[0,n]}|\om(t)-\om'(t)|}{1+\sup_{t\in[0,n]}|\om(t)-\om'(t)|}, \qquad \om,\om'\in \Omega^\dag.$$ Define the coordinate functions $\clp_t(\omega)=\omega(t)$ for all $t\geq 0$ and all $\omega \in \Omega^{\dag}$. For each $t\geq 0$, define $\fil^{\dag}_t:=\sigma\{\clp_s;0\leq s\leq t\}$ and define a $\sigma$-algebra on $\Omega^{\dag}$ by $\fil^\dag=\vee_{t\geq 0}\fil^{\dag}_t$. Let $\mathcal{B}$ denote the Borel $\sigma$-algebra on $\Omega^\dag$ and define the induced probabilities
\begin{align}\label{eq:defcanonpepsintermsofp}
\mbbP^{\eps,n,\dag}(A)=\mbbP^{\eps,n}\{\zedaux{n} \in A\}, \qquad \mbbP^{\eps,\dag}(A)=\mbbP^{\eps}\{\zed \in A\}, \qquad  A \in \mathcal{B}.
\end{align}
Let $C_b(\Omega^\dag)$, equipped with sup norm, be the space of bounded continuous functions on $\Omega^\dag$.
Let $\sppm(\Omega^\dag)$  be the space of probability measures on $\Omega^\dag$ equipped with $weak^*$ topology when $\sppm(\Omega^\dag)$ is considered as dual of $C_b(\Omega^\dag)$.
\end{defn}

\begin{rmk}
The metric space $(\Omega^\dag,D)$ is Polish and the convergence induced by the metric $D$ is uniform convergence on compacts. Also, $\fil^\dag=\mathcal{B}$. The topology on $\sppm(\Omega^\dag)$ is same as the one induced by the Prohorov metric. See for example \cite{MDP}. $\hfill \qed$
\end{rmk}

\begin{thm}\label{thm:main}
Let $\mbbP^{n,\dag}$ be the unique solution of the martingale problem for $\igen^\dag_n$, with initial condition $\zedx_0$ such that $\Phi \zedx_0=\pj_0x$. 
As $\eps \to 0$, the measures $\mbbP^{\eps,n,\dag}$ tend to $\mbbP^{n,\dag}$. 
\end{thm}
\begin{proof}
We follow the approach of \cite{Duke} (also see \cite{Kushbook}).

Three bounded derivatives of $F$ ensure that the coefficients $(a^{(n)},b^{(n)})$ defined in remark \ref{rmk:drifdiffcoeffLndag} have two bounded derivatives. Further, $(a^{(n)},b^{(n)})$ are bounded on $\R^2$. By corollary 6.3.3 of \cite{MDP} the solution of martingale problem for $\igen^\dag_n$ is well-posed. In particular, the solution of martingale problem for $\igen^\dag_n$ with initial condition $\zedx_0$ exists and is unique.

Let $g^{\eps}:\Sp_n \to \R$ be such that it is bounded and has continuous bounded derivatives with respect to $\zedx$, $\tau$ and $y$ (in Frechet sense) and lies in the domain of $\noisegent$. Define
\begin{equation}\label{defMtgeps}
M^{g^{\eps}}_t:=g^{\eps}(\zedaux{n}_t,\teps_t, y^{\eps,n}_t,\xi^\eps_t)-g^{\eps}(\zedaux{n}_0,\teps_0, y^{\eps,n}_0,\xi^\eps_0)-\int_0^t(\igen^{\eps}g^{\eps})(\zedaux{n}_s,\teps_s, y^{\eps,n}_s,\xi^\eps_s)ds,
\end{equation}
then $M^{g^{\eps}}_t$ is a martingale with respect to the filtration $\fil^{\eps,n}_t$.

For  {$g \in C^3(\R^2;\R)$} bounded with bounded derivatives, let $g^\eps=g + \eps \psi^{(1)}_g + \eps^2 \psi^{(2)}_g$ where  $\psi^{(k)}_g$ are ``correctors'' as given in theorem \ref{thm:apprxofgenacttestfunc}. 
Applying \eqref{defMtgeps} to this $g^\eps$ and using equation \eqref{eq:thm:apprxofgenacttestfunc} we have
\begin{align}\label{eq:martpropforgeps}
g(\zedaux{n}_t)-g(\zedaux{n}_s)-\int_s^t(\igen^{\dag}_ng)(\zedaux{n}_u)du\,&=\,\big(M^{g^{\eps}}_t-M^{g^{\eps}}_s\big)\\ \notag 
&\qquad +\eps \int_s^t(\igen^{\eps}\psi^{(2)}_g)(\zedaux{n}_u,\teps_u, y^{\eps,n}_u,\xi^\eps_u)du \\ \notag 
& \qquad - \sum_{k=1}^2 \eps^k\big(\psi^{(k)}_g(\zedaux{n}_t,\teps_t, y^{\eps,n}_t,\xi^\eps_t)-\psi^{(k)}_g(\zedaux{n}_s,\teps_s, y^{\eps,n}_s,\xi^\eps_s)\big),
\end{align}
where $M^{g^{\eps}}_t$ is a $\fil^{\eps,n}_t$ martingale. By theorem  \ref{thm:apprxofgenacttestfunc}, $\psi^{(k)}_g$ and $\igen^{\eps}\psi^{(2)}_g$ are bounded. Hence we have 
\begin{equation}\label{eq:useful2tightnessaux}
|g(\zedaux{n}_t)-g(\zedaux{n}_s)| \leq |\int_{s}^{t}\igen^\dag_n g(\zedaux{n}_u)du| + |M^{g^\eps}_t-M^{g^\eps}_s|+\eps C_1+\eps C_2|t-s|+\eps^2 C_3.
\end{equation}
Let $H_{g^{\eps}}(\zedx,\tau,y,\xi)=\left(\igen^{\eps}(g^{\eps})^2-2g^{\eps}\igen^{\eps}g^{\eps}\right)(\zedx,\tau,y,\xi)$. Then $\la {M}^{g^\eps} \ra_t= \int_0^t H_{g^{\eps}}(\zedaux{n}_s,\teps_s,y^{\eps,n}_s,\xi^\eps_s)ds$. Direct computation yields 
\begin{align*}
H_{g^{\eps}}&=[\noisegent(\psi^{(1)}_g)^2-2\psi^{(1)}_g\noisegent\psi^{(1)}_g] \\
 &+2\eps[\noisegent(\psi^{(1)}_g\psi^{(2)}_g)-\psi^{(1)}_g\noisegent\psi^{(2)}_g-\psi^{(2)}_g\noisegent\psi^{(1)}_g] + \eps^2[\noisegent(\psi^{(2)}_g)^2-2\psi^{(2)}_g\noisegent\psi^{(2)}_g]
\end{align*}
showing that $H_{g^{\eps}}$ is a bounded function. Write the inequality \eqref{eq:useful2tightnessaux} for $g(\zedx)=(\zedx)_1$ and $g(\zedx)=(\zedx)_2$.   
Squaring both the inequalities and adding them and taking expectations, then using
$$\expt^{\eps,n}[|M^{g^\eps}_t-M^{g^\eps}_s|^2\,|\,\fil^{\eps,n}_s]=\expt^{\eps,n}[\la M^{g^\eps}\ra_t -\la M^{g^\eps}\ra_s| \fil^{\eps,n}_s]=\int_s^t\expt^{\eps,n}[H_{g^{\eps}}(\zedaux{n}_u,\teps_u,y^{\eps,n}_u,\xi^\eps_u)| \fil^{\eps,n}_s] du$$
together with the fact that $H_{g^{\eps}}$ is bounded, it can be shown that
$$\varlimsup_{\delta \downarrow 0}\,\, \varlimsup_{\eps\downarrow 0}\,\,\sup_{|t-s|\leq\delta}\,\,\sup \expt^{\eps,n}\left[||\zedaux{n}_t-\zedaux{n}_s||^2 \,|\,\fil^{\eps,n}_s\right]=0$$
where the sup next to the expectation is over the past up to time $s$.
This proves that the family $\mbbP^{\eps,n,\dag}$ is relatively compact (see pages 15-17 of \cite{Duke}). Hence $\mbbP^{\eps,n,\dag}$ have atleast one cluster point.

Multiply equation \eqref{eq:martpropforgeps} by any continuous functional $\Theta_{s}$ of $\zedaux{n}$ which is $\fil^{\eps,n}_s$ measurable, take expectation and passing to limits we obtain
$$\lim_{\eps\downarrow 0}\expt^{\eps,n} \left[\left(g(\zedaux{n}_t)-g(\zedaux{n}_s)-\int_{s}^{t}\igen^\dag_n g(\zedaux{n}_u)du\right)\Theta_{s}\right]=0.$$
Then, it follows that any cluster point of $\mbbP^{\eps,n,\dag}$ solves the martingale problem associated with $\igen^\dag_n$. Since the solution of martingale problem for $\igen^\dag_n$ is unique we have the stated result.
\end{proof}

\begin{thm}\label{thm:main2}
Let $(a,b)$ be as defined remark \ref{rmk:drifdiffcoeffLndag} and let 
\begin{align}\label{eq:infgenLimitEpsZed}
\igen^\dag = \sum_{i=1}^2b_i(\zedx)\frac{\partial}{\partial \zedx_i} + \frac12\sum_{i,j=1}^2a_{ij}(\zedx)\frac{\partial^2}{\partial \zedx_i\partial \zedx_j}.
\end{align}
Let $\mbbP^{\dag}$ be the unique solution of the martingale problem for $\igen^\dag$, with initial condition $\zedx_0$ such that $\Phi \zedx_0=\pj_0x$. 
As $\eps \to 0$, the measures $\mbbP^{\eps,\dag}$ tend to $\mbbP^{\dag}$. 
\end{thm}
\begin{proof}
Atmost linear growth and bounded derivatives of $F$ ensure that the conditions of Theorem 10.2.2 of \cite{MDP} are satisfied. By that theorem, the martingale problem for $\igen^\dag$ with initial condition $\zedx_0$ has unique solution $\mbbP^{\dag}$. 

Define the stopping time $\stopt_n(\om):=\inf\{t\geq 0\,:\,||\om(t)||\geq n\}$ for $\om \in \Om^\dag$. Then, by corollary 10.1.2 in \cite{MDP}, $\mbbP^{\dag}$ agrees with $\mbbP^{n,\dag}$ on $\fil^\dag_{\stopt_n}$. Also, $\mbbP^{\eps,\dag}$ agrees with $\mbbP^{\eps, n,\dag}$ on $\fil^\dag_{\stopt_n}$. By theorem \ref{thm:main} we have $\mbbP^{\eps,n,\dag}\xrightarrow{\eps\to 0} \mbbP^{n,\dag}$, and so by lemma 11.1.1 in \cite{MDP}, $\mbbP^{\eps,\dag}\xrightarrow{\eps\to 0} \mbbP^{\dag}$.
\end{proof}


\section{Effect of the deterministic perturbations $G$ and $G_q$}\label{sec:GGqExpErgNoi}
Following the same route as in proof of theorem \ref{thm:apprxofgenacttestfunc} we find that the effect of $G$ is to add one more drift term $b^{G}$ to the drift coefficient $b$ of $\igen^{\dag}$. The coefficient $b^{G}$ is given by
\begin{align}\label{eq:diftofGExpErgNoi}
b^{G}_i(\zedx) = \frac{1}{2\pi/\om_c}\int_0^{2\pi/\om_c}d\tau\, G(\Phi e^{B\tau}\zedx) \big(e^{-\tau B}\Psiz\big)_i.
\end{align}

\begin{assumpt}\label{ass:assumptonGqcentering}
We assume that $G_q$ satisfies the following centering condition.
$$\frac{1}{2\pi/\om_c}\int_0^{2\pi/\om_c}d\tau\, G_q(\Phi e^{B\tau}\zedx) \big(e^{-\tau B}\Psiz\big)_i \,\,=\,\,0, \qquad i=1,2.$$
\end{assumpt}
Following the same route as in proof of theorem \ref{thm:apprxofgenacttestfunc} we find that the effect of $G_q$ is to add one more drift term $b^{G_q}$ to the drift coefficient $b$ of $\igen^{\dag}$. The coefficient $b^{G_q}$ is given by 
\begin{align}\label{eq:diftofGqExpErgNoi}
b^{G_q}=b^{G_q,P}+b^{G_q,Q},
\end{align}
\begin{align*}
b^{G_q,P}_i(\zedx) &= \frac{1}{2\pi/\om_c}\int_0^{2\pi/\om_c}d\tau\int_{\tau}^{2\pi/\om_c}  G_q(\Phi e^{\tau B}\zedx)\bigg(\big(\Phi e^{(u-\tau)B}\Psiz \big).\nabla\bigg)G_q(\Phi e^{u B}\zedx)\,\big(e^{-u B}\Psiz\big)_i\,du,\\
b^{G_q,Q}_i(\zedx) &= \frac{1}{2\pi/\om_c}\int_0^{2\pi/\om_c}d\tau\int_0^\infty  G_q(\Phi e^{\tau B}\zedx)\bigg(\big(\Th(s)(I-\pih)\Ind \big).\nabla\bigg)G_q(\Phi e^{(\tau+s)B}\zedx)\,\big(e^{-(\tau+s)B}\Psiz\big)_i\,ds.
\end{align*}


\section{Result of this article}\label{sec:MainResultExpErgNoi}

We summarize the result of this paper.

\begin{thm}\label{thm:mainthmExpErgNoimainres}
Let $x$ be as governed by \eqref{eq:detDDE_pert_gennoise_intro}, with coefficients $F, G, G_q$ satisfying assumptions \ref{assmp:assmp_on_FGGq} and \ref{ass:assumptonGqcentering}, and noise $\xi$ satisfying assumption \ref{assmp:assmp_on_noise}. Define $\R^2$-valued process $z$ by $\pi\pj_t x = \Phi z(t)$ where $\pi:\C \to P$ is the projection onto the (critical) subspace $P\subset \C$. Define $\zed_t=z(t/\eps^2)$. Then the probability law of $\zed$ converges as $\eps \to 0$ to the law of diffusion process (with intial condition $\zedx_0=z(0)$) governed by the partial differential operator \eqref{eq:infgenLimitEpsZed} where the diffusion coefficient $a$ is given at \eqref{eq:diffcoeffExpErgNoi} and the drift coefficient $b$ equals 
$$b = b^{F,P}+b^{F,Q}+b^{G}+b^{G_q}, $$
where $b^{F,P}$ and $b^{F,Q}$ are defined at \eqref{eq:driftcoeffbFPExpErgNoi}--\eqref{eq:driftcoeffbFQExpErgNoi} and $b^G$ is defined at \eqref{eq:diftofGExpErgNoi} and $b^{G_q}$ is defined at \eqref{eq:diftofGqExpErgNoi}. The function $R$ that appears in the formulas for $(a,b)$ is the auto-correlation function of the noise. The function $R$ is defined at \eqref{def:autocorrNoiseExpErgNoi}.
\end{thm}

For results regarding vector-valued DDE without proofs, see \cite{PRE}. Usefulness of the above results is also illustrated using numerical simulations in \cite{PRE}. In \cite{PRE} we choose a different basis for $\Phi$ and hence the notation is different from here. We concern with $\mathcal{H}^\eps_t := \frac12||\zed_t||^2$ in \cite{PRE} rather than individual components of $\zed$.

\end{document}